\documentclass[twoside,11pt,reqno]{amsart}
\usepackage{graphicx}
\usepackage{latexsym}
\usepackage{amssymb}
\usepackage{amsmath,bbm}
\usepackage{enumerate}
\usepackage{layout,color,a4wide}
\usepackage[colorlinks]{hyperref}
\usepackage{eufrak}
\usepackage[draft]{fixme}
\newtheorem{prop}{Proposition}
\newtheorem{lemma}{Lemma}

\newtheorem{corollary}{Corollary}
\newtheorem{theorem}{Theorem}
\theoremstyle{definition}
\newtheorem{remark}{Remark}
\newtheorem{example}{Example}

\def\rmd{\mathrm{d}}

\date{\today}

\begin{document}
\title[Quadratic variation of the sum of two Hermite processes]{Asymptotic behavior of the quadratic variation of the sum of two Hermite processes of consecutive orders }
\author{M. Clausel}
\address{M. Clausel, Laboratoire Jean Kuntzmann\\
Universit\'e de Grenoble--Alpes, CNRS\\
F38041 Grenoble Cedex 9}
\email{marianne.clausel@imag.fr}

\author{F. Roueff}
\address{Institut Mines--Telecom, Telecom ParisTech, CNRS LTCI, 46 rue Barrault\\
               75634 Paris Cedex 13, France}
\email{roueff@telecom-paristech.fr}

\author{M.~S. Taqqu}
\address{M.~S. Taqqu, Departement of Mathematics and Statistics, Boston University\\
  Boston, MA 02215, USA}
\email{murad@math.bu.edu}

\author{C. Tudor}
\address{C.  Tudor, Laboratoire Paul Painlev\'e, UMR 8524 du CNRS, Universit\'e
  Lille 1, 59655 Villeneuve d'Ascq, France and Department of Mathematics, Academy of Economical Studies, Bucharest, Romania.}
\email{Ciprian.Tudor@math.univ-lille1.fr}



\maketitle

\begin{abstract}
  Hermite processes are self--similar processes with stationary increments
  which appear as limits of normalized sums of random variables with long range
  dependence. The Hermite process of order $1$ is fractional Brownian motion
  and the Hermite process of order $2$ is the Rosenblatt process.  We consider
  here the sum of two Hermite processes of order $q\geq 1$ and $q+1$ and of
  different Hurst parameters. We then study its quadratic variations at
  different scales. This is akin to a wavelet decomposition. We study both the
  cases where the Hermite processes are dependent and where they are
  independent. In the dependent case, we show that the quadratic variation,
  suitably normalized, converges either to a normal or to a Rosenblatt
  distribution, whatever the order of the original Hermite processes.
\end{abstract}
\section{Introduction}\label{s:intro}
The (centered) quadratic variation of a process $\{Z_t,t\geq 0\}$ is usually defined as
\begin{equation}\label{1}
V_{N}(Z) = \sum_{i=0} ^{N-1} \left[(Z_{t_{i+1}} -Z_{t_{i}})^{2}- \mathbb{E}(Z_{t_{i+1}} -Z_{t_{i}})^{2}\right]\;.
\end{equation}
where $0=t_0<t_1<\cdots<t_N$.  The quadratic variation plays an important role
in the analysis of a stochastic process, for various reasons. For
example, for Brownian motion and martingales, the limit of the sequence
(\ref{1}) is an important element in the It\^o stochastic calculus.  Another
field where the asymptotic behavior of (\ref{1}) is important is estimation
theory: for self-similar processes the quadratic variations are used to
construct consistent estimators for the self-similarity
parameter. The limit in distribution of the sequence $V_{N}$
yields the asymptotic behavior of the associated estimators (see
e.g. \cite{guyon87}, \cite{lang:istas:1997},\cite{lang01}, \cite{coeurjolly:2001}, \cite{taqqu:zhu:2005},~\cite{tudor:2013},
\cite{tudor:viens:2009}). Quadratic variations (and their generalizations) are
also crucial in mathematical finance (see e.g. \cite{BGS}), stochastic analysis
of processes related with fractional Brownian motion (see e.g. \cite{CNW},
\cite{NNT}) or numerical schemes for stochastic differential equations (see
e.g. \cite{neuenkirch:nourdin:2007}). Variations of sums of independent Brownian motion
and fractional Brownian motion are considered in \cite{dozzi13}. The asymptotic behavior of the
quadratic variation of a single Hermite process has been studied in
\cite{chronopoulou:tudor:viens:2011}.

Our purpose is to study the asymptotic behavior of the quadratic variation of a
sum of two dependent Hermite processes of consecutive orders. One could consider other
combinations. We focus on this one because it already displays interesting
features. It shows that the quadratic variation, suitably normalized, converges
either to a normal or to a Rosenblatt distribution, whatever the order of the
original Hermite processes. This would not be the case if only one Hermite
process of order at least equal to two were considered, since then the limit
would always be a Rosenblatt distribution. This would also not be the case if
one considered the sum of two independent Hermite processes. We show indeed that in
the independent case, the quadratic variation asymptotically behaves as that of
a single Hermite process.

We will thus take the process $Z$ in (\ref{1}) to be
$$
Z= Z ^{q,H_{1}}+ Z^{q+1, H_{2}}\;,
$$
where $Z^{q,H}$ denotes a Hermite process of order $q\geq 1$ and with
self-similarity index $ H\in \left( \frac{1}{2}, 1\right)$.  Hermite processes
are self-similar processes with stationary increments and exhibit long-range
dependence.  The Hermite process of order $q\geq 1$ can be written as a
multiple integral of order $q$ with respect to the Wiener process and thus
belongs to the Wiener chaos of order $q$.

We will consider an interspacing
\[
t_{i}-t_{i-1}=\gamma_N
\]
which may depend on $N$. The interspacing $\gamma_N$ may be fixed (as in a time
series setting), grow with $N$ (large scale asymptotics) or decrease with $N$
(small scale asymptotics). The case $\gamma_N=1/N$ is referred to as
\emph{in-fill} asymptotics.
From now on, the expression of $V_N(Z)$ reads
\begin{equation}\label{e:qv}
V_N(Z)=\sum_{i=0} ^{N-1} \left[(Z_{\gamma_N(i+1)} -Z_{\gamma_N i})^{2}- \mathbb{E}(Z_{\gamma_N (i+1)} -Z_{\gamma_N i})^{2}\right]\;.
\end{equation}
Such an interspacing was also considered in~\cite{taqqu:zhu:2005} when studying
the impact of the sampling rate on the estimation of the parameters of
fractional Brownian motion. Since we consider here the sum of two self-similar processes,
one with self-similarity  index $H_1$, the other with self-similarity  index $H_2$,
we expect to find several regimes depending on the growth or decay of $\gamma_N$
with respect to $N$. It seems indeed reasonable to expect that, if $H_1>H_2$
the first process will dominate at large scales and be negligible at small
scales, and the opposite if $H_1<H_2$. Our analysis will in fact exhibit an
intermediate regime between these two. When $H_1=H_2$, it is not clear
whether one term should or should not dominate the other one.

The quadratic variation of the sum $Z=X+Y$ can obviously be decomposed into the
sum of the quadratic variations of $X$ and $Y$ and the so-called quadratic
covariation of $X$ and $Y$ which is defined by
\[
V_{N}(X,Y):=\sum_{i=0} ^{N-1}(X_{t_{i+1}}-X_{t_{i}})(Y_{t_{i+1}}-Y_{t_{i}})
\]
with $0=t_{0}<t_{1}<\dots<t_{N}$. The quadratic covariation shall play a
central role in our analysis. The case where $X= Z ^{H_{1},q}$ and $Y=Z
^{H_{2},q+1}$ are Hermite processes of consecutive orders, exhibits an
interesting situation. If the two processes are independent (that is, they are
expressed as multiple integrals with respect to independent Wiener processes),
then the quadratic covariation of the sum is always dominated by one of the two
quadratic variations $V_{N}(X)$ or $V_{N}(Y)$. On the other hand, surprisingly,
we highlight in this paper that when the two processes are dependent (they can
be written as multiple integrals with respect to the same Wiener process), then
it is their quadratic covariation which may determine the asymptotic behavior
of $V_{N}(X+Y)$. We also find that there is a range of values for the
interspacing $\gamma_N$ where the limit is Rosenblatt in the independent case
and Gaussian in the dependent case. The range includes the choice $\gamma_N=1$
for a large set of $(H_1,H_2)$, as illustrated by the domain $\nu_1<0$ in
Figure~\ref{fig:nu1sign}.

A primary motivation for our work involves the analysis
of wavelet estimators. Of particular interest is the case $H_{2}=2H_{1}-1$
which is related to an open problem in
\cite{clausel:roueff:taqqu:tudor:2012}. See Example~\ref{ex4} for details.

The paper is organized as follows. Section~\ref{prel} contains some
preliminaries on Hermite processes and their properties.  The main results are
stated in Section~\ref{s:main}. The asymptotic behavior of $V_N(Z)$ is given
and illustrated in Section~\ref{sec:asympt-behav-quadr}. The proofs of the main theorem and
propositions are given in Section~\ref{s:proof:VN3} while
Section~\ref{s:lemmas} contains some technical lemmas. Basic facts about
multiple It\^o integrals are gathered in Appendix~\ref{s:appendix}.

\section{Preliminaries}
\label{prel}
Recall that a process $\{X_t,t\geq 0\}$ is self--similar with index $H$ if for
any $a>0$, $\{X_{at},t\geq 0\}$ has the same finite-dimensional distributions
as $\{a^H X_t,t\geq 0\}$. Hermite processes $\{Z^{H,q}_t,t\geq 0\}$, where
$H\in (1/2,1)$, $q=1,2,\cdots$ are self--similar processes with stationary
increments. They appear as limits of normalized sums of random variables with
long--range dependence. The parameter $H$ is the self--similar parameter and
the parameter $q$ denotes the order of the process. The most common Hermite
processes are the fractional Brownian motion $B^H=Z^{H,1}$ (Hermite process of
order $1$) and the Rosenblatt process $R^H=Z^{H,2}$ (Hermite process of order
$2$). Fractional Brownian motion (fBm) is Gaussian but all the other Hermite
processes are non--Gaussian. On the other hand, because of self--similarity and
stationarity of the increments, they all have zero mean and the same
covariance
\[
\mathbb{E}\left[Z^{H,q}_{t_1}Z^{H,q}_{t_2}\right]=\frac{1}{2}\left[|t_1|^{2H}+|t_2|^{2H}-|t_1-t_2|^{2H}\right]\;,
\]
hence $\mathbb{E}\left[(Z^{H,q}_{1})^2\right]=1$. Consequently, the covariance of their increments decays slowly to zero as the lag tends to infinity, namely
\[
\mathbb{E}\left[(Z^{H,q}_{t+1}-Z^{H,q}_{t})(Z^{H,q}_{s+t+1}-Z^{H,q}_{s+t})\right]\sim H(2H-1)s^{2H-2}\mbox{ as }s\to\infty\;.
\]
Observe that the sum over $s\geq 1$ of these covariances diverges, which is an
indication of ``long--range'' dependence.

The Hermite processes $\{Z^{H,q}_{t},t\geq 0\}$ can be represented by Wiener--It\^{o} integrals (see Appendix~\ref{s:appendix} for more details about stochastic integrals), namely
\begin{equation}\label{e:zq}
Z^{H, q} _{t}=c(H,q)\;I_q(L_{t}^{H, q}):=c(H,q)\;\int_{\mathbb{R}^{q}}'L_{t}^{H, q}(y_{1},\cdots, y_{q})\rmd B_{y_{1}}\cdots \rmd B_{y_{q}}\;,
\end{equation}
where  $c(H,q)$ is a positive normalizing constant, $B$ represents standard
Brownian motion and where the kernel is defined by
\begin{equation}\label{e:kern}
L_{t}^{H, q}(y_{1},\cdots, y_{q})= \int_{0} ^{t}(u-y_{1})_{+} ^{-(\frac{1}{2}+\frac{1-H}{q})}\cdots (u-y_{q})_{+}^{-(\frac{1}{2}+\frac{1-H}{q})}\rmd u\;.
\end{equation}
The prime on the integral~(\ref{e:zq})
indicates that one does not integrate over the ``diagonals'', where at least
two entries of the vector $(y_1,\dots,y_q)$ are equal.
Observe that the kernel $L_{t}^{H, q}$ is symmetric and has a finite
$L^2(\mathbb{R}^q)$ norm $\|L_t^{H,q}\|_2<\infty$ because $H\in (1/2,1)$. The
Hermite process $\{Z^{H,q}_{t},t\geq 0\}$ is then well--defined. It has mean
zero and variance
\[
\mathbb{E}\left[(Z^{H,q}_{t})^2\right]=c^2(H,q)\;q!\|L_t^{H,q}\|_2^2\;.
\]
In order to standardize the Hermite process, the positive normalizing constant
$c(H,q)$ is defined by
\begin{equation}
  \label{eq:norm-constant}
   c(H, q) = \left(q!\|L_1^{H,q}\|_2^2\right)^{-1/2}\;,
\end{equation}
so that $\mathbb{E}\left[(Z^{H,q}_{t})^2\right]=t^{2H}$ for all $t\geq0$.

The  fractional Brownian motion is obtained by setting $q=1$ and denoted by
\[
B^H_t=Z^{H,1}_{t}:=c(H, 1)\int_{\mathbb{R}}\left(\int_{0} ^{t}(u-y)_{+}^{H-3/2}\rmd u\right)\rmd B_{y_{1}}\;,
\]
while the Rosenblatt process is obtained by setting $q=2$  and denoted by
\[
R^H_t=Z^{H,2}_{t}:=c(H,2)\int_{\mathbb{R}^{2}}'\left(\int_{0}^{t}(u-y_{1})_{+} ^{H-1}(u-y_{2})_{+}^{H-1}\rmd u\right)\rmd B_{y_{1}}\rmd B_{y_{2}}\;.
\]
The (marginal) distribution $B^H_1$ of the standard fractional Brownian motion
$B^H_t=Z^{H,1}_t$ when $t=1$ is $\mathcal{N}(0,1)$ and the distribution $R^H_1$
of the standard Rosenblatt process $R^H_t$ when $t=1$ is called the Rosenblatt
distribution, see \cite{rosenblatt63} and \cite{veillette13} for more
information about that distribution. The normal distribution and the Rosenblatt
distribution will appear in the limit.

The asymptotic behavior of $V_N(X)$ where $X$ is an Hermite process, namely
$X=Z^{H,q}$ was studied in~\cite{chronopoulou:tudor:viens:2011}. The limit is
either the normal distribution or the Rosenblatt distribution. The normal
distribution appears in the limit when $X$ is the fractional Brownian motion
$Z^{H,1}$ with $H\in (1/2,3/4)$. The Rosenblatt distribution appears in the
limit when $X$ is the fractional Brownian motion with $H\in (3/4,1)$ or when
$Z^{H,q}$ is a Hermite process with $q\geq 2$ and $H\in (1/2,1)$. See
Theorem~\ref{thm:VN12} below, for a precise statement.

We shall focus on the simplest mixed model based on Hermite processes, that is,
\begin{equation}\label{e:mixed}
Z_{t}=Z_{t}^{H_{1},H_{2}}= Z_{t}^{H_{1},q} + Z_{t}^{H_{2}, q+1}\;,
\end{equation}
where $q\geq 1$ and $H_1,H_2\in (1/2,1)$. Processes of the type~(\ref{e:mixed})
appear naturally in the framework of long range dependent Gaussian subordinated
processes (see~\cite{clausel:roueff:taqqu:tudor:2012} and Example \ref{ex4}
below). Observe that~:
\begin{itemize}
\item $Z^{H,q}$ and $Z^{H,q+1}$ are defined in~(\ref{e:zq}) using the same
  underlying Brownian motion $B$ but different kernels $L_t$ are involved.
\item It follows from the previous point that $Z^{H,q}$ and $Z^{H,q+1}$ are
  uncorrelated but dependent, see \cite{bai13}.
\item $Z^{H_{1},H_{2}}$ is not self--similar anymore if $H_1\neq H_2$ but still
  has stationary increments.
\item In the quadratic variations~(\ref{e:qv}) cross--terms
\[
\left(Z^{H_1,q}_{\gamma_N(i+1)}-Z^{H_1,q}_{\gamma_N i}\right)\left(Z^{H_2,q+1}_{\gamma_N(i+1)}-Z^{H_2,q+1}_{\gamma_N i}\right)\;,
\]
will appear. We will show that their (renormalized) partial sum  is asymptotically normal.
\end{itemize}
The notation $\overset{(d)}{\longrightarrow}$ refers to the  convergence in
distribution and $a_{N}\ll  b_{N} $  means that $a_{N}=o(b_{N})$ as $N\to
\infty$ and $a_N\sim b_N$ means $a_N/b_N\to1$ as $N\to\infty$. The notation $X_N=o_P(1)$ means that $X_N\rightarrow 0$ in probability.

\section{Main results}\label{s:main}

\subsection{Main assumptions}

Throughout the paper, we consider $H_1,H_2\in(1/2,1)$ and an integer $q\geq1$,
\begin{equation}\label{e:vnq}
V_{N}:= V_N(Z) = \sum_{i=0} ^{N-1} \left[(Z_{t_{i+1}} -Z_{t_{i}} ) ^{2}- \mathbb{E}(Z_{t_{i+1}} -Z_{t_{i}} ) ^{2}\right]\;,
\end{equation}
where $t_i=\gamma_N i$ and $Z$ is the sum of the two Hermite processes
$Z^{H_1,q}$ and $Z^{H_2,q+1}$ as defined in~(\ref{e:mixed}).
The sum $V_{N}$ will be split into three terms as follows
\begin{equation}\label{e:splitVNq}
V_{N} = V_{N}^{(1)} + V_{N}^{(2)} + 2V_{N}^{(3)}\;,
\end{equation}
where
\begin{equation}
V_{N}^{(1)}= \sum_{i=0}^{N-1} \left[\left(Z^{H_{1},q}_{t_{i+1}}-Z^{H_{1}, q}_{t_{i}}\right)^{2}-\mathbb{E}\left(Z^{H_{1},q}_{t_{i+1}}-Z^{H_{1},q}_{t_{i}}\right)^{2}\right]\;,
\end{equation}
\begin{equation}
 V_{N}^{(2)} =\sum_{i=0}^{N-1} \left[\left(Z^{H_{2},q+1}_{t_{i+1}}-Z^{H_{2},q+1}_{t_{i}}\right)^{2}-\mathbb{E}\left(Z^{H_{2},q+1}_{t_{i+1}}-Z^{H_{2},q+1}_{t_{i}}\right)^{2}\right]\;,
\end{equation}
and
\begin{equation}\label{e:VN3}
V_{N}^{(3)}=\sum_{i=0}^{N-1}\left(Z^{H_{1},q}_{t_{i+1}}-Z^{H_{1},q} _{t_{i}}\right)\left(Z^{H_{2},q+1}_{t_{i+1}}-Z^{H_{2},q+1}_{t_{i}}\right)\;.
\end{equation}
The mean of the cross--term~(\ref{e:VN3}) vanishes because the terms in the
product are Wiener--It\^{o} integrals of different orders and hence are
uncorrelated (see formula (\ref{e:isom})).  We further denote the corresponding
standard deviations by
\begin{equation}
  \label{eq:sig2def}
  \sigma_N:=\left(\mathbb{E}\left[V_N^2\right]\right)^{1/2}\quad\text{and}\quad
  \sigma^{(i)}_N:=\left(\mathbb{E}\left[(V_N^{(i)})^2\right]\right)^{1/2}\quad\text{for $i=1,2,3$.}
\end{equation}
\subsection{Asymptotic behavior of $V_{N}^{(1)}$, $V_{N}^{(2)}$ and $V_{N}^{(3)}$}
To investigate the asymptotic behavior of $V_N$ and $\sigma_N$, we shall
consider the terms $V_{N}^{(1)}$, $V_{N}^{(2)}$ and $V_{N}^{(3)}$ separately,
without any assumption on the scale sequence $(\gamma_N)$.

First we recall well--known results about the asymptotic behavior of the
sequences $V_{N}^{(1)}$ and $V_{N}^{(2)}$.
\begin{theorem}\label{thm:VN12}
Denote
\begin{equation}
  \label{eq:h1def}
  h_1=\max\left(\frac{1}{2},1-\frac{2(1-H_1)}{q}\right)=
  \begin{cases}
    \frac12&\text{ if $q=1$ and $H_1\leq 3/4$}\\
1-\frac{2(1-H_1)}{q}&\text{ if $q\geq2$ or $H_1\geq 3/4$}
  \end{cases}
  \;,
\end{equation}
and
\begin{equation}
  \label{eq:deltadef}
\delta=\mathbbm{1}_{\{H_1=3/4\}\cap\{q=1\}}\;,
\end{equation}
that is, $\delta=1$ if $H_1=3/4$ and $q=1$, and $\delta=0$ if $H_1\neq3/4$ or
$q\geq2$. Then as $N\to\infty$,
\begin{equation}
  \label{e:sigma1}
 \sigma_N^{(1)}\sim a(q,H_1)\;\gamma_N^{2H_1}\;N^{h_1}\,(\log N)^{\delta/2} \;.
\end{equation}
Moreover, we have the following asymptotic limits as $N\to\infty$.
\begin{enumerate}
\item
\begin{enumerate}
\item \label{1a} If $q=1$ and $H_1\in (1/2,3/4]$, then
\begin{equation}
  \label{e:glim1a}
\frac{V_N^{(1)}}{\sigma^{(1)}_N}\overset{(d)}{\longrightarrow}\mathcal{N}(0,1)
\;,
\end{equation}
\item \label{1b} If  $q=1$ and $H_1\in (3/4,1)$ or if $q\geq 2$ and $H_1\in (1/2,1)$, then
\begin{equation}
  \label{e:ros2a}
\frac{V_N^{(1)}}{\sigma^{(1)}_N}\overset{(d)}{\longrightarrow}R_1^{1-2(1-H_1)/q}
\end{equation}
\end{enumerate}
\item \label{2} If $q\geq 1$ (that is, $q+1 \geq 2$) and $H_2\in (1/2,1)$, then
\begin{equation}
  \label{e:ros2b}
\frac{V_N^{(2)}}{\sigma^{(2)}_N}\overset{(d)}{\longrightarrow}R_1^{1-2(1-H_2)/(q+1)}
\quad\text{with}\quad
\sigma_N^{(2)}\sim a(q+1,H_2)\;N^{[1-2(1-H_2)/(q+1)]}\gamma_N^{2H_2}\;.
\end{equation}
\end{enumerate}
Here $a(q,H_1)$ and $a(q+1,H_2)$ are positive
constants.
\end{theorem}
\begin{proof}
Point~(\ref{1a}) goes back to~\cite{breuer:major:1983} and
Point~(\ref{1b}) with $q=1$ and $H_1\in (3/4,1)$ goes back to
\cite{taqqu:1975}.
Point~(\ref{1b}) with $q\ge 2$ and $H_1\in (1/2,1)$
  can  be deduced
from~\cite{chronopoulou:tudor:viens:2011} (see Theorem~1.1 and its proof).

For the expression of the constant $a(q,H_1)$ in (\ref{e:sigma1}) with $q=1$, that is for $a(1,H_1)$, see
Propositions 5.1, 5.2 and 5.3  in ~\cite{tudor:2013} with $ H\in (1/2,3/4)$,  $H\in (3/4,1)$ and $H=3/4$ respectively.
For the expression of $a(q,H_1)$ with $q\ge 2$ and $H_1\in (1/2,1)$, see Proposition 3.1 in
\cite{chronopoulou:tudor:viens:2011}.
The expression of the constant $a(q+1,H_2)$ follows from that of $a(q,H_1)$.

The exponent of $\gamma_N$ in~(\ref{e:sigma1}) and~(\ref{e:ros2b}) results from
the fact that $Z^{H_1,q}$ and $Z^{H_2,q+1}$ are self-similar with index $H_1$
and $H_2$, respectively.
\end{proof}
In view of the decomposition ~(\ref{e:splitVNq}) and
of Theorem \ref{thm:VN12}, we need to investigate the asymptotic behavior of
the cross--term $V_N^{(3)}$ in order to get the asymptotic behavior of $V_N$.
\begin{theorem}\label{thm:VN3}
We have the following convergence and asymptotic equivalence as $N\to\infty$.
\[
\frac{V_N^{(3)}}{\sigma^{(3)}_N}\overset{(d)}{\longrightarrow}\mathcal{N}(0,1)
\quad\text{with}\quad
\sigma_N^{(3)}\sim b(q,H_1,H_2)\;N^{1-(1-H_{2})/(q+1)}\gamma _{N}^{H_{1}+H_{2}}\;,
\]
where $b(q,H_1,H_2)$ is a positive constant.
\end{theorem}
\begin{remark}\label{rem:gene1}
Theorem~\ref{thm:VN3} cannot be directly extended to the general case where the  process $Z$ is the sum of two Hermite processes of order $q_1,q_2$ with $q_2-q_1>1$. This is because the proof is based on the fact that $V_N^{(3)}$ admits a Gaussian leading term. This may not happen if $q_2-q_1>1$ (see the proof of Proposition~\ref{pro:NT} and Remark~\ref{rem:gene2} for more details).
In contrast, Theorems \ref{thm:indep-case} and \ref{thm:Vn-independent-case} below can be easily extended.
\end{remark}
\begin{proof}
The proof of Theorem~\ref{thm:VN3} is found in Section~\ref{s:proof:VN3}.
\end{proof}

\subsection{Quadratic covariation in the independent case}

Theorem~\ref{thm:VN3} will imply that the term $V^{(3)}_{N}$, which corresponds
to the quadratic covariation of $Z^{H_{1},q}$ and $ Z^{H_{2},q+1}$, may
dominate in the asymptotic behavior of the sequence $V_{N}$. On the other hand,
if the two Hermite processes are independent, the quadratic covariation is
always dominated by the quadratic variation of one of these processes. This is
a consequence of the following theorem.
\begin{theorem}\label{thm:indep-case}
  Assume that for every $t\geq 0$, $Z_{t}^{H_{1},q}$ and $Z_{t}^{H_{2},q+1}$
  are given by~(\ref{e:zq}) and define
$$
\tilde V^{(3)}_{N}= \sum_{i=0}^{N-1}\left(Z^{H_{1},q}_{t_{i+1}}-Z^{H_{1},q}
  _{t_{i}}\right)\left(\tilde Z^{H_{2},q+1}_{t_{i+1}}-\tilde Z^{H_{2},q+1}_{t_{i}}\right)\;,
$$
where the process $\tilde{Z}^{H_{2},q+1}$ is an independent copy of $Z^{H_{2},q+1}$.
Then, as $N\to \infty$,
\begin{equation*}
\mathbb{E}\left[ \left( \tilde V^{(3)}_{N} \right)^{2}\right] =o\left( \sigma_N^{(1)}\times\sigma_N^{(2)}\right)\;,
\end{equation*}
where $\sigma_N^{(1)}$ and $\sigma_N^{(2)}$ are defined by~(\ref{eq:sig2def}).
\end{theorem}
\begin{proof} We have, from the independence of the two Hermite processes,
\begin{eqnarray*}
\mathbb{E}\left[ \left(\tilde V^{(3)}_{N}\right) ^{2}\right] &=& \sum_{i,j=0}
^{N-1} \gamma_{i,j} (Z ^{H_{1},q})\gamma _{i,j}(\tilde Z^{H_{2},q+1} ) \;,
\end{eqnarray*}
where
$$
\gamma_{i,j}(X):=
\mathbb{E}\left[(X_{t_{i+1}}-X_{t_{i}})(X_{t_{j+1}}-X_{t_{j}})\right]\;.
$$
Since the covariance structure of the Hermite process $Z^{H,q}$ is the same for all $q\geq 1$, we obtain
\begin{eqnarray*}
\mathbb{E}\left[\left(\tilde V^{(3)}_{N}\right)^{2}\right] &=& \sum_{i,j=0}^{N-1} \gamma_{i,j}(Z^{H_{1},1})\gamma _{i,j}(Z^{H_{2},1} ) \\
&\leq &\left( \sum_{i,j=0}^{N-1} \left(\gamma_{i,j} (Z^{H_{1},1})\right)^{2}\right)^{\frac{1}{2}} \left(\sum_{i,j=0}^{N-1}\left(\gamma _{i,j}(Z^{H_{2},1})\right)^{2}\right)^{\frac{1}{2}}\;,
\end{eqnarray*}
where the last line follows from the Cauchy--Schwarz inequality.
Recall that for two jointly centered Gaussian random variables  $X$ and $Y$, we have
 $$
 \left( \mathrm{Cov} (X,Y) \right) ^{2}= \frac{1}{2} \mathrm{Cov} (X^{2}, Y^{2} ).
 $$
  Hence
\begin{eqnarray*}
\sum_{i,j=0}^{N-1} \left( \gamma_{i,j} (Z^{H_{1},1})\right)^{2}  &=&
\sum_{i,j=0}^{N-1}  \left(\mathrm{Cov}\left( Z^{H_{1},1}_{t_{i+1}}- Z^{H_{1},1}_{t_{i}}, Z^{H_{1},1}_{t_{j+1}}- Z^{H_{1},1}_{t_{j}}\right)\right)^{2}\\&=&
\frac{1}{2}\sum_{i,j=0}^{N-1} \mathrm{Cov}\left( \left( Z^{H_{1},1}_{t_{i+1}}- Z^{H_{1},1}_{t_{i}}\right)^{2},  \left( Z^{H_{1},1}_{t_{j+1}}- Z^{H_{1},1}_{t_{j}}\right)^{2}\right)\\
& =&  \frac{1}{2}\mathrm{Var} \left[ \sum_{i=0}^{N-1}\left(\left( Z^{H_{1},1}_{t_{i+1}}- Z^{H_{1},1}_{t_{i}}\right)^{2}-\mathbb{E}\left[\left( Z^{H_{1},1}_{t_{i+1}}- Z^{H_{1},1}_{t_{i}}\right)^{2}\right]\right)\right]\\
&=&  \frac{1}{2}\mathbb{E} \left[ \left(V_{N} \left( Z ^{1, H_{1}}\right) \right) ^{2} \right]
\end{eqnarray*}
by using the notation (\ref{1}). Consequently,
\begin{eqnarray*}
\mathbb{E}\left[ \left(\tilde V^{(3)}_{N}\right) ^{2}\right] &\leq& \frac{1}{2}
\left\{ \mathbb{E} \left[ \left(V_{N} \left( Z^{H_{1},1}\right)\right) ^{2}
  \right]\mathbb{E}\left[\left(V_{N}\left(Z^{H_{2},1}\right)\right)^{2}\right]\right\}^{\frac{1}{2}}\;.
\end{eqnarray*}
By Theorem \ref{thm:VN12}, we know that, as $N\to\infty$, the rate of
convergence of the variance of the quadratic variations of the Hermite process $Z^{H,q}$  (strictly)
increases with respect to $q$ (when $H$ is fixed).  Therefore, since $q\geq1$
and $q+1>1$,  we get, as $N\to \infty$,
\begin{eqnarray*}
\mathbb{E}\left[ \left(\tilde V^{(3)}_{N}\right) ^{2}\right] =o\left(  \left\{ \mathbb{E} \left[ \left(V_{N} \left( Z ^{H_{1},q}\right)\right) ^{2}
  \right]
\mathbb{E}\left[ \left( V_{N} \left( Z ^{H_{2},q+1}\right) \right) ^{2}
\right]\right\}^{\frac{1}{2}}\right)\;,
\end{eqnarray*}
which concludes the proof.
\end{proof}

\section{Asymptotic behavior of the quadratic variation of the sum}
\label{sec:asympt-behav-quadr}
\subsection{Dependent case}
It is now clear that the asymptotic behavior of $V_N$ will depend on the
relative behavior of the three sumands $V_{N}^{(1)}$, $V_{N}^{(2)}$ and
$V_{N}^{(3)}$. More precisely we have the following result.

\begin{theorem}\label{th:main}
Let us define
\begin{equation}\label{e:nu}
\nu_1:=\frac{1-H_2}{1+q}-1+\max\left(\frac12,1-\frac{2(1-H_1)}{q}\right)
\;<\;\frac{1-H_2}{1+q}=:\nu_2\;.
\end{equation}
Denoting $\delta$ as in~(\ref{eq:deltadef}), we
have the following asymptotic equivalence as $N\to\infty$~:
 \begin{enumerate}
  \item\label{item:V1} If $\gamma_N^{H_2-H_1}\ll N^{\nu_1}\,(\log N)^{\delta/2}$ then
$$
V_N=V_N^{(1)} \; (1+o_P(1))\;.
$$
\item\label{item:V3} If $N^{\nu_1}\,(\log N)^{\delta/2}\ll \gamma_N^{H_2-H_1}\ll N^{\nu_2}$, then
$$
V_N=2V_N^{(3)} \; (1+o_P(1))\;.
$$
\item\label{item:V2} If $\gamma_N^{H_2-H_1}\gg N^{\nu_2}$ then
$$
V_N=V_N^{(2)} \; (1+o_P(1))\;.
$$
\end{enumerate}
\end{theorem}
\begin{proof}
Let us compare the terms $V_{N}^{(1)}, V_{N}^{(2)} $ and $V_{N}^{(3)}$ in
each case considered in Theorem~\ref{th:main}.
By Theorem~\ref{thm:VN12} and Theorem~\ref{thm:VN3}, we have, for
some positive constants $c_1,\,c_2$ and $c_3$,
\begin{align*}
\mathbb{E}\left[\left|V_{N}^{(1)}\right|^2\right]^{1/2} &\sim c_1\,\gamma_N^{2H_1}\,\,N^{h_1}\,(\log N)^{\delta/2} \;,\\
\mathbb{E}\left[\left|V_{N}^{(2)}\right|^2\right]^{1/2} &\sim
c_2\,N^{1-2(1-H_2)/(q+1)} \gamma_N^{2H_2}\;,\\
\mathbb{E}\left[\left|V_{N}^{(3)}\right|^2\right]^{1/2} &\sim c_3\,N^{1-(1-H_2)/(q+1)} \gamma_N^{H_1+H_2}\;,
\end{align*}
and these rates always correspond to the rate of convergence in distribution.

The different cases are obtained by using the definitions of
$\nu_1<\nu_2$ in~(\ref{e:nu}) and by computing the following ratios of the above
rates for $V_{N}^{(1)}$ versus  $V_{N}^{(3)}$~:
\begin{equation}
  \label{eq:1vs3}
\frac{\gamma_N^{2H_1}\,N^{h_1}\,(\log N)^{\delta/2}}{N^{1-(1-H_2)/(q+1)}\gamma_N^{H_1+H_2}}
=\frac{N^{\nu_1}\,(\log N)^{\delta/2}}{\gamma_N^{H_2-H_1}}\;,
\end{equation}
and $V_{N}^{(2)}$ versus  $V_{N}^{(3)}$~:
\begin{equation}
  \label{eq:2vs3}
\frac{N^{1-2(1-H_2)/(q+1)} \gamma_N^{2H_2}}{N^{1-(1-H_2)/(q+1)}
  \gamma_N^{H_1+H_2}}
=\frac{\gamma_N^{H_2-H_1}}{N^{\nu_2}}\;.
\end{equation}
Observing that $\nu_1<\nu_2$ and thus $N^{\nu_1}(\log N)^{\delta/2}\ll N^{\nu_2}$, we get in the three cases~:
\begin{enumerate}
\item If
$$
\gamma_N^{H_2-H_1}\ll N^{\nu_1}\,(\log N)^{\delta/2},
$$
 then $V_N^{(1)}$ dominates $V_N^{(3)}$
  by~(\ref{eq:1vs3}). But since it implies $\gamma_N^{H_2-H_1}\ll N^{\nu_2}$, we have
  by~(\ref{eq:2vs3}) that   $V_N^{(3)}$ dominates $V_N^{(2)}$. Hence
  $V_N^{(1)}$ dominates in this case.
\item   If
$$
N^{\nu_1}\,(\log N)^{\delta/2}\ll \gamma_N^{H_2-H_1}\ll N^{\nu_2},
$$
 then  $V_N^{(3)}$ dominates both
  $V_N^{(1)}$ and $V_N^{(2)}$ by~(\ref{eq:1vs3}) and~(\ref{eq:2vs3}), respectively.
\item If
$$
\gamma_N^{H_2-H_1}\gg N^{\nu_2},
$$
 then  $V_N^{(2)}$ dominates $V_N^{(3)}$
  by~(\ref{eq:2vs3}). But since it implies $\gamma_N^{H_2-H_1}\gg N^{\nu_1}\,(\log N)^{\delta/2}$, we have
  by~(\ref{eq:1vs3}) that   $V_N^{(3)}$ dominates $V_N^{(1)}$. Hence
  $V_N^{(2)}$ dominates in this case.
\end{enumerate}
This concludes the proof  of Theorem~\ref{th:main}.
\end{proof}
\begin{remark}\label{rem:nu1sign}
  Note that $\nu_2>0$ but $\nu_1$ can be positive, zero, or negative. In fact,
$$
\nu_1=
\begin{cases}
  \frac{1-H_2}{1+q}-\frac12&\text{ if $q=1$ and $H_1<3/4$}\\
  \frac{1-H_2}{1+q}-\frac{2(1-H_1)}{q}&\text{ otherwise.}
\end{cases}
$$
It follows that
$$
\nu_1\leq0 \Longleftrightarrow H_2\geq1-\frac{2(q+1)(1-H_1)}{q}\;,
$$
with equality on the left-hand side if and only if there is equality on the
right-hand side. The equality case corresponds to having $(H_1,H_2)$ on the segment
with end points $(1-q/(4(q+1)),1/2)$ and $(1,1)$, see
Figure~\ref{fig:nu1sign}. The $H_1$ coordinate of the
bottom end point is   $1-q/(4(q+1))$. For $q=1$ it equals
$1-1/8=0.875$ and, as $q\to\infty$, it decreases towards $3/4$.
\end{remark}

\begin{figure}[h]
  \centering
  \includegraphics[width=0.8\textwidth]{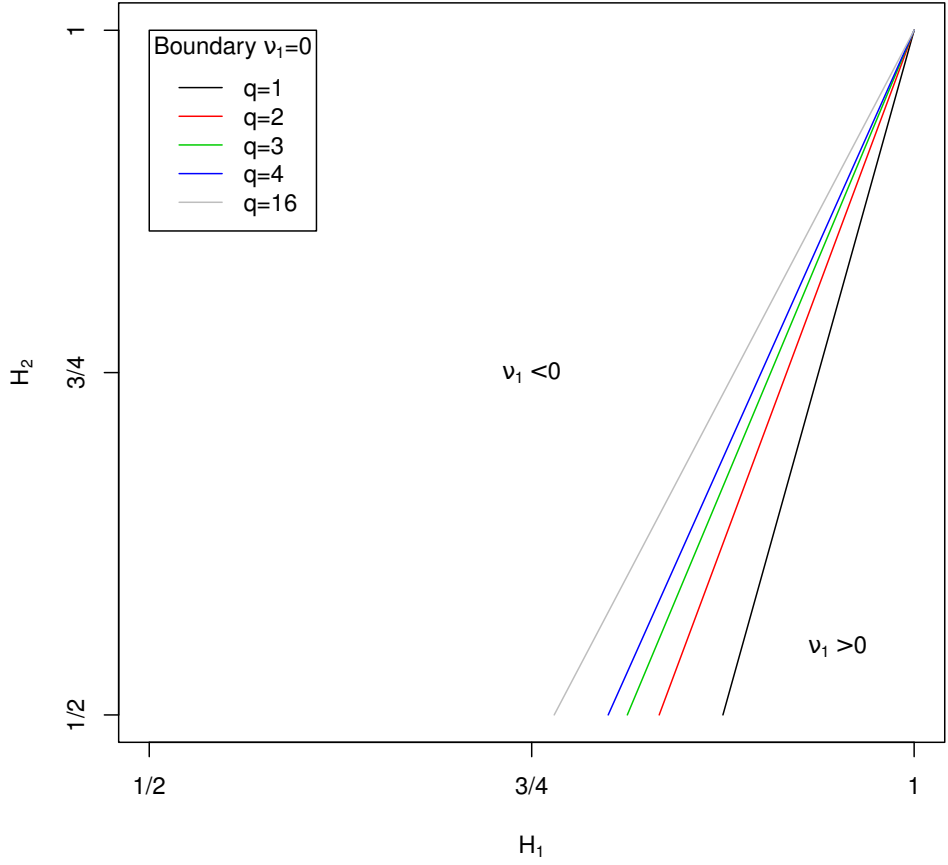}
  \caption{Domains of points $(H_1,H_2)$ where the signs of $\nu_1$ is negative
    or positive. The lines show the boundary between these sets. The right-hand
    line corresponds to $q=1$ and the left-hand line to $q=16$. The processes
    here are dependent.}
  \label{fig:nu1sign}
\end{figure}
Let us illustrate Theorem~\ref{th:main} with some examples.

\begin{example}\label{ex2}
  In the particular case where $H_{1}=H_{2}$, by Remark~\ref{rem:nu1sign}, we
  always have $\nu_1<0$. It follows that we are in Case~(\ref{item:V3}) of Theorem~\ref{th:main}
  whatever the values of $q=1,2,\dots$ and the interspacing scale $\gamma_N$.
  Thus, the dominant part of $V_{N}$ is the summand $2V_{N} ^{(3)}$. By
  Theorem~\ref{thm:VN3}, we conclude that the limit of the normalized quadratic
  variation of the sum of two Hermite processes with the same self-similarity
  index and successive orders is asymptotically Gaussian.
\end{example}

\begin{example}\label{ex:gammaeq1}
  When $\gamma_N=1$, the asymptotic behavior of the quadratic variation depends
  on the sign of $\nu_1$. If $\nu_1<0$, we are in Case~(\ref{item:V3}) of
  Theorem~\ref{th:main}, the dominant part of $V_{N}$ is the summand $2V_{N}
  ^{(3)}$ and the limit is asymptotically Gaussian by Theorem~\ref{thm:VN3}. If
  $\nu_1>0$, we are in Case~(\ref{item:V1}) of Theorem~\ref{th:main}, the
  dominant part of $V_{N}$ is the summand $V_{N}^{(1)}$ and by
  Theorem~\ref{thm:VN12}, the limit is Rosenblatt. Indeed, $\nu_1>0$ excludes
  $q=1$ and $H_1\leq 3/4$, see Figure~\ref{fig:nu1sign}.
\end{example}



\begin{example}\label{ex4}
  The case $H_2=2H_1-1$ and $q=1$ is of special interest because it is related
  to an open problem in \cite{clausel:roueff:taqqu:tudor:2012}.  Let us recall
  the context.  Suppose you have unit variance Gaussian stationary data
  $Y_i,i\geq 1$ with spectral density $f(\lambda)$ which blows up like
  $|\lambda|^{-2d}$ at the origin, with $1/4<d<1/2$. Then the partial sums
  behave asymptotically like fractional Brownian motion with index $H_1$, where
\begin{equation}\label{e:link1}
2H_1=(2d-1)+2=2d+1\;.
\end{equation}
On the other hand, referring to \cite{clausel:roueff:taqqu:tudor:2011a}, the
partial sums of $Y_i^2-1,i\geq 1$ behave like a Rosenblatt process with index
$H_2$, where
\begin{equation}\label{e:link2}
2H_2=2(2d-1)+2=4d\;.
\end{equation}
It follows that, conveniently normalized, for $n$ large,
$\sum_{k=1}^{[nt]}(Y_k+Y_k^2-1)$ can be seen as a process $Z_t$ as defined
by~(\ref{e:mixed}) with $H_2=2H_1-1$.

\medskip

Applying Theorem~\ref{th:main} with $H_2=2H_1-1$ and $q=1$, we obtain the
following result.
\begin{corollary}\label{cor:specialcase}
If $H_2=2H_1-1$ and $q=1$, we have the following asymptotic equivalence as $N\to\infty$~:
\begin{enumerate}
\item If $\gamma_N\gg  N$ then
$$
V_N=V_N^{(1)} \; (1+o_P(1))\;.
$$
\item If $N^{-1}\ll \gamma_N\ll N$, then
$$
V_N=2V_N^{(3)} \; (1+o_P(1))\;.
$$
\item If $\gamma_N\ll N^{-1}$ then
$$
V_N=V_N^{(2)} \; (1+o_P(1))\;.
$$
\end{enumerate}
\end{corollary}
\begin{proof}
Observe that if $H_2=2H_1-1$ and $q=1$, one has $H_2-H_1=H_1-1<0$. In addition,
the expression of the two exponents $\nu_1,\nu_2$  in~(\ref{e:nu}) can be simplified as follows~:
\[
\nu_1=\max\left(\frac{1}{2}-H_1,H_1-1\right)=-\min\left(H_1-\frac{1}{2},1-H_1\right) \quad \mbox{ and } \quad\nu_2=1-H_1.
\]
Then $\nu_1/(H_2-H_1)=\min\left(1,\frac{H_1-1/2}{1-H_1}\right)$ and
$\nu_2/(H_2-H_1)=-1$. Moreover observe that $H_2=2H_1-1>1/2$ implies
$H_1>3/4$ which in turns implies
$\min\left(1,\frac{H_1-1/2}{1-H_1}\right)=1$ and $\delta=0$. Thus $\nu_1/(H_2-H_1)=1$ and
$\nu_2/(H_2-H_1)=-1$.  Corollary~\ref{cor:specialcase} then follows from
Theorem~\ref{th:main}.
\end{proof}
\end{example}

\subsection{Independent case}

Theorem~\ref{th:main} should be contrasted with the following result involving
independent processes.
\begin{theorem}\label{thm:Vn-independent-case}
Assume that the process $\tilde{Z}^{H_{2},q+1}$ is an independent copy of $Z^{H_{2},q+1}$.
Let
$$
\tilde
V_N=V_N({Z}^{H_{1},q}+\tilde{Z}^{H_{2},q+1})
$$
 and
 $$
 \tilde V_N^{(2)}=V_N(\tilde{Z}^{H_{2},q+1}).
 $$
Define $\delta$ as in~(\ref{eq:deltadef}) and $\nu_1,\nu_2$ as in~(\ref{e:nu}).
We have the following asymptotic equivalence as $N\to\infty$~:
 \begin{enumerate}
  \item\label{item:V1indep} If $\gamma_N^{2(H_2-H_1)}\ll N^{\nu_1+\nu_2}\,(\log N)^{\delta/2}$ then
$$
\tilde V_N=V_N^{(1)} \; (1+o_P(1))\;.
$$
\item\label{item:V2indep} If $\gamma_N^{2(H_2-H_1)}\gg N^{\nu_1+\nu_2}\,(\log N)^{\delta/2}$ then
$$
\tilde V_N=\tilde V_N^{(2)} \; (1+o_P(1))\;.
$$
\end{enumerate}
\end{theorem}
\begin{proof}
  By Theorem~\ref{thm:indep-case}, we only need to compare $V_N^{(1)}=V_N({Z}^{H_{1},q})$
  and $\tilde V_N^{(2)}=V_N(\tilde{Z}^{H_{2},q+1})\overset{d}{=}V_N({Z}^{H_{2},q+1})=V_N^{(2)}$.  Using
  Theorem~\ref{thm:VN12} as in the proof of Theorem~\ref{th:main}, the ratio between~(\ref{eq:1vs3})
  and~(\ref{eq:2vs3}), gives that, as $N\to\infty$,
$$
\frac{\sigma_N^{(1)}}{\sigma_N^{(2)}}\sim c\;\frac{N^{\nu_1+\nu_2}\,(\log N)^{\delta/2}}{\gamma_N^{2(H_2-H_1)}}\;,
$$
where $c$ is a positive constant.
This concludes the proof  of Theorem~\ref{thm:Vn-independent-case}.
\end{proof}

\begin{remark}\label{rem:nu1pnu2sign}
  Note that $\nu_1+\nu_2$ can be positive, zero, or negative. In fact,
$$
\nu_1+\nu_2=
\begin{cases}
  \frac{2(1-H_2)}{1+q}-\frac12&\text{ if $q=1$ and $H_1<3/4$}\\
  \frac{2(1-H_2)}{1+q}-\frac{2(1-H_1)}{q}&\text{ otherwise.}
\end{cases}
$$
It follows that
$$
\nu_1+\nu_2\leq0 \Longleftrightarrow H_2\geq1-\frac{(q+1)(1-H_1)}{q}\;,
$$
with equality on the left-hand side if and only if there is equality on the
right-hand side. The equality case corresponds to having $(H_1,H_2)$ on the segment
with end points $(1-q/(2(q+1)),1/2)$ and $(1,1)$, see
Figure~\ref{fig:nu1pnu2sign}. The $H_1$ coordinate of the
bottom end point is   $1-q/(2(q+1))$. For $q=1$ it equals
$3/4$ and, as $q\to\infty$, it decreases towards $1/2$. In contrast with the
dependent case described in Remark~\ref{rem:nu1sign}, the bottom of the boundary lines are
pushed to the left, with half the slopes, compare Figures~\ref{fig:nu1sign} and~\ref{fig:nu1pnu2sign}.
\end{remark}

\begin{figure}[h]
  \centering
  \includegraphics[width=0.8\textwidth]{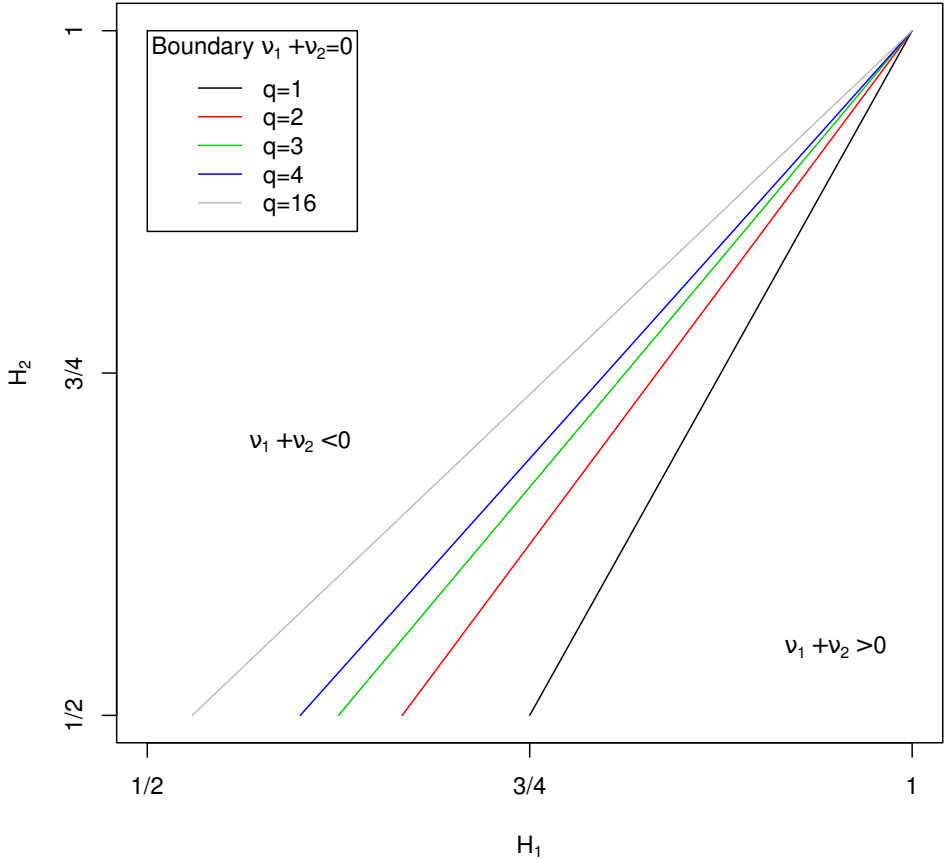}
  \caption{Domains of points $(H_1,H_2)$ where the signs of $\nu_1+\nu_2$ is negative
    or positive. The lines show the boundary between these sets. The right-hand
    line corresponds to $q=1$ and the left-hand line to $q=16$. The processes
    here are independent.}
  \label{fig:nu1pnu2sign}
\end{figure}


We now illustrate Theorem~\ref{thm:Vn-independent-case} where the sum of two independent
processes is considered.
\begin{example}\label{exple:h1eqh2-indep}
 If $H_1=H_2$, by Remark~\ref{rem:nu1pnu2sign}, we always have
 $\nu_1+\nu_2<0$. Thus the dominant part of $\tilde V_N$ is always $\tilde
 V_N^{(2)}$. By
  Theorem~\ref{thm:VN12}, we conclude that the limit of the normalized quadratic
  variation of the sum of two \emph{independent} Hermite processes with the same self-similarity
  index and successive orders is asymptotically Rosenblatt.
\end{example}
\begin{example}\label{exple:gammaeq1-indep}
  When $\gamma_N=1$, the asymptotic behavior of the quadratic variation depends
  on the sign of $\nu_1+\nu_2$. If $\nu_1+\nu_2<0$, we are in
  Case~\ref{item:V2indep} of Theorem~\ref{thm:Vn-independent-case}, the
  dominant part of $V_{N}$ is $\tilde V_{N} ^{(2)}$ and the limit is
  asymptotically Rosenblatt by Theorem~\ref{thm:VN12}. If $\nu_1+\nu_2>0$, we
  are in Case~\ref{item:V1indep} of Theorem~\ref{thm:Vn-independent-case}, the
  dominant part of $V_{N}$ is $V_{N}^{(1)}$ and by Theorem~\ref{thm:VN12}, the
  limit is Rosenblatt. Indeed, $\nu_1+\nu_2>0$ excludes the case $q=1$ and $H_1\leq 3/4$, see
  Figure~\ref{fig:nu1pnu2sign}.
\end{example}
\begin{remark}
In Examples~\ref{ex2} and~\ref{exple:h1eqh2-indep}, we considered the
setting $H_1=H_2$ in the dependent and independent cases. We see that the
corresponding limits always differ, it is Gaussian in the dependent case and it
is Rosenblatt in the independent case.

The contrast between Examples~\ref{ex:gammaeq1} and
~\ref{exple:gammaeq1-indep}, which both correspond to the setting $\gamma_N=1$
is a bit more involved. If $\nu_1>0$, then $\nu_1+\nu_2>0$ and we have the same
asymptotic behavior in both cases and the asymptotic limit is Rosenblatt.
On the other hand, if $\nu_1<0$, then, in
the dependent case we have a Gaussian limit and in the independent case we
again have a Rosenblatt limit.
\end{remark}

\section{Proof  of Theorem~\ref{thm:VN3}}\label{s:proof:VN3}
The proof of Theorem~\ref{thm:VN3} is based on its decomposition in Wiener chaos of the cross term $V_N^{(3)}$. We first need some notation~:
For any $q$ and $(H_1,H_2)\in (1/2,1)^2$, set
\begin{equation}\label{e:star:indices}
H_1^*(q)=\frac{1-H_{1}}{q}+\frac{1-H_{2}}{q+1}\;.
\end{equation}
The function $\widetilde{\beta}_{a,b}$ will appear as part of the kernel
involved in the Wiener chaos expansion of $V_N^{(3)}$. It is defined
on $\mathbb{R}^2\setminus\{(u,v),u=v\}$ for any $a,b>-1$ such that $a+b<-1$ as~:
\begin{equation}\label{e:betat}
\widetilde{\beta}_{a,b}(u,v)=\left\{\begin{array}{l}\beta(a+1,-1-a-b)\mbox{ if }u<v,\\
\beta(b+1,-1-a-b)\mbox{ if }v<u\;,\end{array}\right.
\end{equation}
where $\beta$ denotes the beta function
\[
\beta(x,y)=\int_0^1 t^{x-1}(1-t)^{y-1}\rmd t=\frac{\Gamma(x)\Gamma(y)}{\Gamma(x+y)},\,x,y>0\;.
\]
\begin{prop}\label{pro:WienerChaos}
The sum $V_{N}^{(3)}$ admits the following expansion into Wiener chaos~:
\begin{equation}\label{e:WienerChaos}
V_{N}^{(3)}= \sum_{k=0} ^{q} V_{N} ^{(3,k)}\;,
\end{equation}
where for every $k=0,\cdots, q$,
\begin{equation}\label{e:VN3k}
V_{N}^{(3,k)}=M(k,q,H_{1},H_{2})I_{2q+1-2k}\left(\sum_{i=0} ^{N-1}f_{N,i}^{(k)}\right)\;,
\end{equation}
with
\begin{eqnarray*}
f_{N,i}^{(k)}(y_1,\cdots,y_{2q+1-2k})&=&\int_{t_{i}} ^{t_{i+1}} \int_{t_{i}}^{t_{i+1}}
\left[\prod_{i=1}^{q-k}(u-y_{i}) _{+} ^{-(\frac{1}{2}+ \frac{1-H_{1}}{q})}\right]\left[\prod_{i=q-k+1}^{2q+1-2k}(v-y_{i}) _{+} ^{-(\frac{1}{2}+ \frac{1-H_{2}}{q+1})}\right] \\
&&\times\left[\widetilde{\beta}_{-(\frac{1}{2}+ \frac{1-H_{1}}{q}),-(\frac{1}{2}+ \frac{1-H_{2}}{q+1})}(u,v)\right]^k\vert u-v\vert ^{-k H_1^*(q)}\rmd u\rmd v\;,
\end{eqnarray*}
where $\widetilde{\beta}_{a,b}$ has been defined in~(\ref{e:betat}) and,
defining $c(H,q)$ as in~(\ref{eq:norm-constant}),
\[
M(k,q, H_{1}, H_{2} )= c(H_{1}, q) c(H_{2}, q+1)k!{{q}\choose{k}}{{q+1}\choose{k}}\;.
\]
\end{prop}
\begin{proof}
Using the integral expression~(\ref{e:zq}) of the two Hermite processes $Z^{(q,H_1)}$ and $Z^{(q+1,H_2)}$ and by definition~(\ref{e:VN3}) of the sum $V_N^{(3)}$, we get that
\[
\bar V_{N}^{(3)} :=\frac{V_{N}^{(3)}}{c(H_{1}, q) c(H_{2}, q+1)}=\sum_{i=0}^{N-1} I_{q}\left(L_{t_{i+1}}^{H_{1},q}-L_{t_{i}}^{H_{1},q}\right)I_{q+1}\left( L_{t_{i+1}}^{H_{2}, q+1} - L_{t_{i}}^{H_{2}, q+1}\right)\;,
\]
where the two kernels $L_t^{H_1,q}$, $L_t^{H_2,q+1}$ are defined in~(\ref{e:kern}). We now use the product formula~(\ref{e:product}) and deduce that
\begin{align}\nonumber
\bar V_{N}^{(3)}&=\sum_{i=0} ^{N-1}\left[\sum_{k=0} ^{q} k! {{q}\choose{k}} {{q+1}\choose{k}}I_{2q+1-2k}\left((L_{t_{i+1}}^{H_{1},q} - L_{t_{i}} ^{H_{1}, q})\otimes _{k}( L_{t_{i+1}} ^{H_{2}, q+1} - L_{t_{i}} ^{H_{2}, q+1})\right)\right]\\
\label{e:V3N}
&=\sum_{k=0} ^{q} k! {{q}\choose{k}} {{q+1}\choose{k}}I_{2q+1-2k}\left(\sum_{i=0} ^{N-1}(L_{t_{i+1}}^{H_{1},q} - L_{t_{i}} ^{H_{1}, q})\otimes _{k}( L_{t_{i+1}} ^{H_{2}, q+1} - L_{t_{i}} ^{H_{2}, q+1})\right)\;.
\end{align}
\noindent To get an explicit expression for each term
\[
\left(L_{t_{i+1}}^{H_{1},q}-L_{t_{i}}^{H_{1}, q}\right)\otimes _{k}\left(L_{t_{i+1}}^{H_{2},q+1}- L_{t_{i}}^{H_{2},q+1}\right)\;,
\]
we use the definition of the $\otimes_k$ product given in the appendix~:
\begin{eqnarray*}
&&\left[\left(L_{t_{i+1}}^{H_{1},q}-L_{t_{i}}^{H_{1},q}\right)\otimes_{k}\left(L_{t_{i+1}}^{H_{2},q+1}-L_{t_{i}}^{H_{2},q+1} \right)\right] (y_{1},\cdots, y_{2q+1-2k} ) \\
&=&\int_{\mathbb{R}^{k}}\left(L_{t_{i+1}}^{H_{1},q} - L_{t_{i}}^{H_{1},q} \right)(y_{1},\cdots, y_{q-k}, x_{1},\cdots, x_{k})\\
&&\times\left(L_{t_{i+1}}^{H_{2},q+1}-L_{t_{i}}^{H_{2},q+1}\right)(y_{q-k+1},\cdots,y_{2q+1-2k}, x_{1},\cdots, x_{k} )\rmd x_{1}\cdots \rmd x_{k}\;.
\end{eqnarray*}
Using the specific form~(\ref{e:kern}) of the kernel $L_t^{H,q}$ and the Fubini Theorem, the last formula reads
\begin{eqnarray*}
&&\left[\left(L_{t_{i+1}}^{H_{1},q}-L_{t_{i}}^{H_{1},
      q}\right)\otimes_{k}\left(L_{t_{i+1}}^{H_{2},q+1}-L_{t_{i}}^{H_{2},q+1}\right)\right](y_{1},\cdots,
y_{2q+1-2k}) \\
&=&\int_{t_{i}}^{t_{i+1}}\int_{t_{i}}^{t_{i+1}}\left[\prod_{i=1}^{q-k}(u-y_{i})_{+}^{-(\frac{1}{2}+\frac{1-H_{1}}{q})}\right]\left[\prod_{i=q-k+1}^{2q+1-2k}(v-y_{i})_{+}^{-(\frac{1}{2}+\frac{1-H_{2}}{q+1})}\right]\\
&&\times\left[\prod_{i=1}^k\int_{\mathbb{R}}(u-x_i) _{+} ^{-(\frac{1}{2}+ \frac{1-H_{1}}{q})}(v-x_i)_{+}^{-(\frac{1}{2}+ \frac{1-H_{2}}{q+1})}\rmd x_i\right]\rmd u\rmd v\\
&=&\int_{t_{i}}^{t_{i+1}}\int_{t_{i}}^{t_{i+1}}\left[\prod_{i=1}^{q-k}(u-y_{i})_{+}^{-(\frac{1}{2}+ \frac{1-H_{1}}{q})} \right]\left[\prod_{i=q-k+1}^{2q+1-2k}(v-y_{i})_{+}^{-(\frac{1}{2}+\frac{1-H_{2}}{q+1})}\right]\\
&&\times\left[\int_{-\infty} ^{u\wedge v}(u-x)^{-(\frac{1}{2}+ \frac{1-H_{1}}{q})} (v-x)^{-(\frac{1}{2}+\frac{1-H_{2}}{q+1})}\rmd x\right]^{k}\rmd u\rmd v\;.
\end{eqnarray*}
But for any real numbers $a,b>-1$ such that $a+b<-1$, Lemma~\ref{lem:beta:tilde} implies that
\begin{equation}\label{e:fct:beta}
\int_{-\infty}^{u\wedge v}(u-x)^a(v-x)^b\rmd x=\widetilde{\beta}_{a,b}(u,v)|u-v|^{a+b+1}\;,
\end{equation}
where the function $\widetilde{\beta}_{a,b}$ has been defined in~(\ref{e:betat}). Hence
\begin{eqnarray*}
&&\left[  \left( L_{t_{i+1}}^{H_{1}, q} - L_{t_{i}}^{H_{1}, q}\right)\otimes_{k} \left(L_{t_{i+1}}^{H_{2},q+1}- L_{t_{i}}^{H_{2},q+1} \right)\right](y_{1},\cdots, y_{2q+1-2k})\\
&=&\int_{t_{i}}^{t_{i+1}}\int_{t_{i}}^{t_{i+1}}\left[\prod_{i=1}^{q-k}(u-y_{i})_{+}^{-\left(\frac{1}{2}+\frac{1-H_{1}}{q}\right)}\right]
\left[\prod_{i=q-k+1}^{2q+1-2k}(v-y_{i})_{+}^{-\left(\frac{1}{2}+\frac{1-H_{2}}{q+1}\right)}\right]\\
&&\times\left[\widetilde{\beta}_{-(\frac{1}{2}+ \frac{1-H_{1}}{q}),-(\frac{1}{2}+\frac{1-H_{2}}{q+1})}(u,v)\right]^{k}\vert u-v\vert^{-k H_1^*(q)}\rmd u\rmd v\;,
\end{eqnarray*}
where we defined $H_1^*$ in~(\ref{e:star:indices}). Combining this equality and relation~(\ref{e:V3N}) then leads to the decomposition~(\ref{e:WienerChaos}) of the sum $V^{(3)}_N$. This completes the proof of Proposition~\ref{pro:WienerChaos}.
\end{proof}

We now bound the $L^2$--norm of $V^{(3,k)}_N$ for any $k=0,\cdots,q$ and deduce
that the terms $V^{(3,k)}_N$ are for all $k<q$ negligible with respect to
$V^{(3,q)}_N$. Proposition~\ref{pro:NT} below then directly implies
Theorem~\ref{thm:VN3}. We set
\begin{equation}\label{e:varepsilon}
\varepsilon(\alpha)=\left\{\begin{array}{l}1\mbox{ if }\alpha=1,\\0\mbox{ otherwise}.\end{array}\right.
\end{equation}
\begin{prop}\label{pro:NT}
For any $k=0,\cdots,q-1$,
\begin{equation}\label{e:UB:VN3k}
\|V^{(3,k)}_N\|^2\leq\;C\;N^{\left[2-\min\left(\alpha,1\right)\right]}(\log(N))^{\varepsilon(\alpha)}\gamma _{N}^{2H_1+2H_2}\;,
\end{equation}
where $\varepsilon$ is defined in~(\ref{e:varepsilon}),
\begin{equation}
  \label{eq:alpha}
\alpha=2(q-k)(1-H_1)/q+2(q+1-k)(1-H_2)/(q+1)
\end{equation}
and as $N\to\infty$,
\begin{equation}\label{e:UB:VN3q}
N^{-2+\frac{2(1-H_{2}) }{q+1}}\gamma _{N}^{-2(H_1+H_2)}\|V^{(3,q)}_N\|^2\rightarrow  b^2(H_1,H_2,q)\;\;,
\end{equation}
for some $b(H_1,H_2,q)>0$. The leading term is the one with $k=q$. Moreover, $V^{(3,q)}_N$ is a Gaussian random variable, and thus
\[
N^{\frac{(1-H_{2}) }{q+1}-1}\gamma _{N}^{-(H_1+H_2)}V^{(3,q)}_N\overset{(d)}{\longrightarrow}b(H_1,H_2,q)\;\mathcal{N}(0,1)\;.
\]
\end{prop}
\begin{remark}\label{rem:gene2}
The proof of~(\ref{e:UB:VN3q}) is based on the fact that, because  $Z$ is the sum of two Hermite processes of consecutive orders, then $V^{(3,q)}_N$ has a centered Gaussian term in its decomposition (\ref{e:V3N}) and this term turns out  to be the leading term.  We can then deduce its asymptotic behavior from that of its variance. Since the variance of a simple Wiener--It\^{o} integral is related to the $L^2$--norm of the integrand, we obtain (\ref{e:UB:VN3q}). Note that this proof does not extend to the case where $Z$ is the sum of two Hermite processes of order $q_1,q_2$ with  $q_2-q_1>1$. In that case,  there is no Gaussian term in the sum $V^{(3)}_N$ and hence no Gaussian leading term. Thus, Proposition~\ref{pro:NT} cannot be extended in a simple way to more general cases.
\end{remark}
\begin{proof}
  We use the notation of Proposition~\ref{pro:WienerChaos}. If $n\geq 2$,
  we have by~(\ref{e:ineqf}), that $\mathbb{E}[I_n(f)^2]\leq n!\|f\|_2^2$ whereas
  in the case $n=1$, $f$ is trivially symmetric and this inequality becomes an
  equality. We first consider the case $k=0,\cdots,q-1$. By the integral
  definition~(\ref{e:VN3k}) of the terms $V_{N}^{(3,k)}$, we get that
\[
\mathbb{E}\left[\left|V_{N}^{(3,k)}\right|^{2}\right]\leq  M_1^{2}(k,q,H_1,H_2)\int_{\mathbb{R}^{2q+1-2k}}\left[\sum_{i,j=0}^{N-1}f_{N,i}^{(k)}(y)f_{N,j}^{(k)}(y)\right]\rmd y_{1}\cdots \rmd y_{2q+1-2k}\;,
\]
with
$$
M_1^2(k,q,H_1,H_2)=(2q+1-2k)!M^2(k,q,H_1,H_2).
$$
 Using the explicit expression of $f_{N,\ell}^{(k)}$ given for any $\ell=0,\cdots,N-1$ in Proposition~\ref{pro:WienerChaos}, we deduce that~:
\begin{equation}\label{e:ineq:1}
\mathbb{E}\left[\left|V_{N}^{(3,k)}\right| ^{2}\right]\leq  M_1^{2}(k,q,H_1,H_2)\int_{\mathbb{R}^{2q+1-2k}}\left[\sum_{i,j=0}^{N-1} g_{N,i,j}^{(q, k)}(y)\right]\rmd y_{1}\cdots \rmd y_{2q+1-2k}\;,
\end{equation}
with
\begin{eqnarray*}
&&g_{N,i,j}^{(q, k)}(y_1,\cdots,y_{2q+1-2k})\\
&=&\int_{u=t_{i}}^{t_{i+1}}\int_{v=t_{i}}^{t_{i+1}}\int_{u'=t_{j}}^{t_{j+1}}\int_{v'=t_{j}}^{t_{j+1}}\prod_{\ell=1}^{q-k}\left[(u-y_{\ell})_{+}(u'-y_{\ell})_{+}\right]^{-\left(\frac{1}{2}+\frac{1-H_{1}}{q}\right)}\\
&&\times\prod_{\ell=q-k+1}^{2q-2k+1}\left[(v-y_{\ell})_{+}(v'-y_{\ell})_{+}\right]^{-\left(\frac{1}{2}+\frac{1-H_{2}}{q+1}\right)}\\
&&\times\left[\widetilde{\beta}_{-(\frac{1}{2}+ \frac{1-H_{1}}{q}),-(\frac{1}{2}+\frac{1-H_{2}}{q+1})}(u,v)\widetilde{\beta}_{-(\frac{1}{2}+ \frac{1-H_{1}}{q}),-(\frac{1}{2}+\frac{1-H_{2}}{q+1})}(u',v')\vert u-v\vert^{-H_1^*(q)}\vert u'-v'\vert^{-H_1^*(q)}\right]^k\rmd u\rmd v\rmd u'\rmd v'\;.
\end{eqnarray*}
On the other hand, equality~(\ref{e:int:beta:tilde:2}) of Lemma~\ref{lem:beta:tilde} implies that for any $\ell=1,\cdots,q-k$
\[
\int_{y_\ell\in\mathbb{R}}(u-y_{\ell})_{+}^{-(\frac{1}{2}+\frac{1-H_{1}}{q})}(u'-y_{\ell})_{+}^{-(\frac{1}{2}+ \frac{1-H_{1}}{q})}\rmd y_\ell=\beta(a_1+1,-2a_1-1)|u-u'|^{-\frac{(2-2H_{1})}{q}}\;,
\]
with $a_1=-(1/2+(1-H_1)/q)$ and that for any $\ell=q-k+1,\cdots,2q-2k+1$
\[
\int_{y_\ell\in\mathbb{R}}(v-y_{\ell})_{+}^{-(\frac{1}{2}+\frac{1-H_{2}}{q+1})}(v'-y_{\ell})_{+}^{-(\frac{1}{2}+\frac{1-H_{2}}{q+1})}\rmd y_\ell=\beta(a_2+1,-2a_2-1)\vert v-v'\vert^{-\frac{(2-2H_{2})}{q+1}}\;,
\]
with $a_2=-(1/2+(1-H_2)/(q+1))$. Hence, combining the Fubini theorem, inequality~(\ref{e:ineq:1}) and these two last equalities implies that for some $M_2(k,q,H_1,H_2)>0$
\begin{eqnarray*}
\mathbb{E}\left[\left(V_{N}^{(3,k)}\right)^{2}\right] &\leq &
M_2^{2}(k,q,H_1,H_2)\sum_{i,j=0}^{N-1}
\int_{u=t_{i}}^{t_{i+1}}\int_{v=t_{i}}^{t_{i+1}}\int_{u'=t_{j}}^{t_{j+1}}\int_{v'=t_{j}}^{t_{j+1}}
h(u,u',v,v')\rmd u\rmd v\rmd u'\rmd v' \;,
\end{eqnarray*}
with

$$
h(u,u',v,v') = \left[\vert u-u'\vert^{-\frac{2-2H_{1}}{q}}\right]^{q-k}\left[\vert v-v'\vert^{-\frac{2-2H_{2}}{q+1}}\right]^{q-k+1}
$$
$$
\times\left[\widetilde{\beta}_{-(\frac{1}{2}+ \frac{1-H_{1}}{q}),-(\frac{1}{2}+\frac{1-H_{2}}{q+1})}(u,v)\widetilde{\beta}_{-(\frac{1}{2}+ \frac{1-H_{1}}{q}),-(\frac{1}{2}+\frac{1-H_{2}}{q+1})}(u',v')\vert u-v\vert^{-H_1^*(q)}\vert u'-v'\vert^{- H_1^*(q)}\right]^k\;,
$$

and
\[
M_2^2(k,q,H_1,H_2)=M_1(k,q,H_1,H_2)^2\beta(a_1+1,-2a_1-1)^{q-k}\beta(a_2+1,-2a_2-1)^{q-k+1}\;.
\]
Recall that $t_{i}= i\gamma _{N}$ and use the change of variables
\[
U=\gamma _{N}^{-1}(u-i\gamma _{N}),\,V=\gamma _{N}^{-1}(v-i\gamma _{N}),\,U'=\gamma _{N}^{-1}(u'-j\gamma _{N}),\,V'=\gamma _{N}^{-1}(v'-j\gamma _{N})\;.
\]
We have
\begin{eqnarray*}
\mathbb{E}\left[\left(V_{N}^{(3,k)}\right)^{2}\right] &\leq & c\gamma_{N}^{-2k H_1^*(q)}\gamma_{N}^{-\frac{(q-k)(2-2H_{1})}{q}}\gamma_{N}^{-\frac{(q-k+1)(2-2H_{2})}{q+1}}\gamma_{N}^{4}\\
&&\times\sum_{i,j=0}^{N-1} \int_{[0,1] ^{4}}H_{i,j}(U,U',V,V')\rmd U\rmd V\rmd U'\rmd V'\;,
\end{eqnarray*}
with
\begin{eqnarray}
H_{i,j}(U,U',V,V')&=&\left[\vert U-U'+i-j\vert^{-\frac{2(1-H_{1})}{q}}\right]^{q-k}\left[\vert V-V'+i-j\vert^{-\frac{2(1-H_{2})}{q+1}}\right]^{q-k+1}\label{e:Hij}\\
&&\times\left(\widetilde{\beta}_{-(\frac{1}{2}+ \frac{1-H_{1}}{q}),-(\frac{1}{2}+\frac{1-H_{2}}{q+1})}(U,V)\widetilde{\beta}_{-(\frac{1}{2}+ \frac{1-H_{1}}{q}),-(\frac{1}{2}+\frac{1-H_{2}}{q+1})}(U',V')\right)^{k}\\
&&\times\vert U-V\vert^{-kH_1^*(q)}\cdot\vert U'-V'\vert^{-kH_1^*(q)}\nonumber\;,
\end{eqnarray}
since $\widetilde{\beta}_{a,b}(u,v)$ only depends on the sign of $u-v$. Now we simplify the expression involving powers of $\gamma_{N}$. Since
\begin{eqnarray}
&&-2k H_1^*(q)-\frac{(q-k)(2-2H_{1})}{q}-\frac{(q-k+1)(2-2H_{2})}{q+1}+4\nonumber\\
&=&-2k\frac{1-H_{1}}{q}-2k\frac{1-H_{2}}{q+1}-\frac{(q-k)(2-2H_{1})}{q}-\frac{(q-k+1)(2-2H_{2})}{q+1}+4\nonumber\\
&=&\left(-\frac{2q}{q}-\frac{2(q+1)}{q+1}+4\right)+H_1\left(\frac{2k}{q}+\frac{2(q-k)}{q}\right)+H_2\left(\frac{2k}{q+1}+\frac{2(q-k+1)}{q+1}\right)\nonumber\\
&=&2H_1+2H_2\label{e:sumH}\;,
\end{eqnarray}
we deduce that
\begin{equation}\label{e:majo}
\mathbb{E}\left[\left(V_{N}^{(3,k)}\right)^{2}\right] \leq  c \gamma_{N}^{2H_1+2H_2}\sum_{i,j=0}^{N-1}\int_{[0,1]^{4}}H_{i,j}(U,U',V,V')\rmd U \rmd V\rmd U'\rmd V'\;.
\end{equation}
To obtain~(\ref{e:UB:VN3k}), we check that we can apply 
Lemma~\ref{lem:riemman:sum} below with
\[
\alpha_1=\frac{2(q-k)(1-H_{1})}{q},\,\alpha_2=\frac{2(q-k+1)(1-H_{2})}{q+1}\;,
\]
and
\[
F(U,V)=\left[\widetilde{\beta}_{-(\frac{1}{2}+ \frac{1-H_{1}}{q}),-(\frac{1}{2}+\frac{1-H_{2}}{q+1})}(U,V)\vert U-V\vert^{-H_1^*(q)}\right]^{k}\;.
\]
Since $k H_1^*(q)\leq k(2q+1)/(2q(q+1))<1$, Condition~(\ref{e:cvF}) holds.

It
remains to check Condition~(\ref{e:cvFp}). Note that $\widetilde{\beta}_{H_1^*(q),H_2^*(q)}$ is
bounded and then, for some $C>0$,
\[
|F(U,V)|\leq C\,\vert U-V\vert^{-kH_1^*(q)}\;.
\]
We deduce the finiteness of the integrals
\[
\int_{[0,1]^4}|U-U'+\ell|^{-\alpha_1}|V-V'+\ell|^{-\alpha_2}F(U,V)F(U',V')\rmd U\rmd U'\rmd V\rmd V'\;,
\]
as follows~:
\begin{enumerate}
\item if $\ell=0$, we observe that $\alpha_1,\alpha_2,kH_1^*(q)\in (0,1)$ and by~(\ref{e:sumH}),
\[
\alpha_1+\alpha_2+2kH_1^*(q)=2(1-H_1)+2(1-H_2)<3
\]
We then apply Part (1) of Lemma~\ref{lem:cv}.
\item if $\ell=1$ or $\ell=-1$, we observe that $\alpha_1,\alpha_2,kH_1^*(q)\in (0,1)$ and apply Part (2) of Lemma~\ref{lem:cv}.
\item if $|\ell|\geq 2$, we observe that on $[0,1]^2$,
\[
|U-U'+\ell|^{-\alpha_1}|V-V'+\ell|^{-\alpha_2}\leq ||\ell|-1|^{-\alpha_1-\alpha_2}\left(\int_{[0,1]}\vert U-V\vert^{-kH_1^*(q)}\rmd U\rmd V\right)^2<\infty\;,
\]
since $kH_1^*(q)<1$.
\end{enumerate}
This completes the proof of inequality~(\ref{e:UB:VN3k}) in the case $k\in\{0,\cdots,q-1\}$.

Now we consider the case where $k=q$. The approach is exactly the same except that Inequality~(\ref{e:majo}) becomes an equality because in~(\ref{e:VN3k}) $I_{2q+1-2k}=I_1$ becomes a Gaussian integral. One then has
\begin{equation}\label{e:equal}
\mathbb{E}\left[\left(V_{N}^{(3,q)}\right)^{2}\right] = c \gamma_{N} ^{2H_1+2H_2}\sum_{i,j=0}^{N-1}\int_{[0,1]^{4}}H_{i,j}(U,U',V,V')\rmd U \rmd V\rmd U'\rmd V'\;,
\end{equation}
with (see~\ref{e:Hij}),
\begin{eqnarray*}
H_{i,j}(U,U',V,V')&=&\left[\vert V-V'+i-j\vert^{-\frac{2(1-H_{2})}{q+1}}\right]\left[\widetilde{\beta}_{-(\frac{1}{2}+ \frac{1-H_{1}}{q}),-(\frac{1}{2}+\frac{1-H_{2}}{q+1})}(U,V)\vert U-V\vert^{-H_1^*(q)}\right]^q\\
&&\times\left[\widetilde{\beta}_{-(\frac{1}{2}+ \frac{1-H_{1}}{q}),-(\frac{1}{2}+\frac{1-H_{2}}{q+1})}(U',V')\vert U'-V'\vert^{-H_1^*(q)}\right]^q\;.
\end{eqnarray*}

To conclude, we now apply Part~(2) of Lemma~\ref{lem:riemman:sum} with
\[
\alpha_1=0,\,\alpha_2=\frac{2(1-H_{2})}{q+1}\;,
\]
and
\[
F(u,v)=\left[\widetilde{\beta}_{H_1^*(q),H_2^*(q)}(U,V)\vert U-V\vert^{-H_1^*(q)}\right]^q\;,
\]
Since $\alpha_1+\alpha_2<1$, the equality~(\ref{e:UB:VN3q}) follows. Observe
that $\alpha$ in~(\ref{eq:alpha}) decreases with $k$. Therefore
the leading term of the sum~(\ref{e:WienerChaos}) is obtained for $k=q$, that is, the summand in
the first Wiener chaos. Finally, observe that since this term is Gaussian,
convergence of the variance implies convergence in distribution. This completes
the proof of Proposition~\ref{pro:NT} and hence of Theorem~\ref{thm:VN3}.
\end{proof}

\section{Technical lemmas}\label{s:lemmas}
\begin{lemma}\label{lem:beta:tilde}
Consider the special function $\beta$ defined for any $x,y>0$ as
\[
\beta(x,y)=\int_0^1 t^{x-1}(1-t)^{y-1}\rmd t\;.
\]
Define on $\mathbb{R}^2\setminus\{(u,v),u=v\}$,  for any $a,b>-1$ such that $a+b<-1$ the function $\widetilde{\beta}_{a,b}$ as~:
\begin{equation}\label{e:beta:tilde}
\widetilde{\beta}_{a,b}(u,v)=\left\{\begin{array}{l}\beta(a+1,-1-a-b)\mbox{ if }u<v,\\
\beta(b+1,-1-a-b)\mbox{ if }v<u\;.\end{array}\right.
\end{equation}
Then
\begin{equation}\label{e:int:beta:tilde}
\int_{-\infty} ^{u\wedge v}(u-s)^a(v-s)^b\rmd s=\widetilde{\beta}_{a,b}(u,v)|u-v|^{a+b+1}\;.
\end{equation}
In particular,
\begin{equation}\label{e:int:beta:tilde:2}
\int_{-\infty} ^{u\wedge v}(u-s)^a(v-s)^b\rmd s\leq C(a,b)|u-v|^{a+b+1}\;,
\end{equation}
with
\begin{equation}\label{e:CB}
C(a,b)=\sup_{(u,v)\in\mathbb{R}^2}\left[\widetilde{\beta}_{a,b}(u,v)\right]<\infty\;.
\end{equation}
\end{lemma}
\begin{proof}
We use the equivalent definition of function $\beta$
\begin{equation}\label{e:beta:2}
\beta(x,y)=\int_0^{\infty} \frac{t^{x-1}}{(1+t)^{x+y}}\rmd t\;.
\end{equation}
Consider first the case where $u<v$. In the integral $\int_{-\infty} ^{u}(u-s)^a(v-s)^b\rmd s$, we set $s'=(u-s)/(v-u)$. We get
\begin{eqnarray*}
\int_{-\infty} ^{u}(u-s)^a(v-s)^b\rmd s&=&\int_{0}^{\infty}\left[(v-u)^a(s')^a\right]\left[(v-u)^b(1+s')^b\right](v-u)\rmd s'\\
&=&(v-u)^{a+b+1}\int_{0}^{\infty} (s')^{a}(1+s')^b\rmd s'\;.
\end{eqnarray*}
Hence, in view of~(\ref{e:beta:2}), we deduce that
\[
\int_{-\infty} ^{u\wedge v}(u-s)^a(v-s)^b\rmd s=(v-u)^{a+b+1}\beta(x,y)
\]
with $x-1=a$ and $x+y=-b$. This implies~(\ref{e:int:beta:tilde}) in the case
$u<v$. The other case $v<u$ is obtained by symmetry.

The finiteness of the constant $C(a,b)$ results from the fact that by definition of $\widetilde{\beta}$,
\[
\sup_{(u,v)\in\mathbb{R}^2}\widetilde{\beta}(u,v)=\max(\beta(a+1,-1-a-b),\beta(b+1,-1-a-b))\;,
\]
which is finite (since $\beta(x,y)$ is finite for each $x,y>-1$).
\end{proof}
\begin{lemma}\label{lem:riemman:sum}
Let $\alpha_1,\alpha_2\in [0,1)$, $F$ a function defined from $[0,1]^2$ to $\mathbb{R}_+^*$ such that
\begin{equation}\label{e:cvF}
\gamma=\int_{[0,1]^2}F(U,V)\rmd U \rmd V<\infty\;.
\end{equation}
and for any $\ell\in\mathbb{Z}$
\begin{equation}\label{e:cvFp}
\Delta(\ell)=\int_{[0,1]^{4}}\left[\left|U-U'+\ell\right|^{-\alpha_1}\left|V-V'+\ell\right|^{-\alpha_2}\right]F(U,V)F(U',V')\rmd U \rmd V\rmd U'\rmd V'<\infty\;.
\end{equation}
Then
\begin{enumerate}
\item if $\alpha_1+\alpha_2\geq1$, there exists some $C>0$ such that
\begin{equation}\label{e:ineq:lem:2}
\sum_{i,j=0}^{N-1}\Delta(i-j)\leq CN\log(N)^{\varepsilon(\alpha_1+\alpha_2)}\;,
\end{equation}
where $\varepsilon$ has been defined in~(\ref{e:varepsilon}).
\item if $\alpha_1+\alpha_2< 1$, we have
\begin{equation}\label{e:lim:lem:2}
\lim_{N\to\infty}\left(\frac{N^{\alpha_1+\alpha_2}}{N^2}\left[\sum_{i,j=0}^{N-1}\Delta(i-j)\right]\right)= \frac{2\gamma^2}{(1-\alpha_1-\alpha_2)(2-\alpha_1-\alpha_2)}\;.
\end{equation}
\end{enumerate}
\end{lemma}
\begin{proof}
We first observe that $\Delta(\ell)=\Delta(-\ell)$ for all $\ell\in\mathbb{Z}$
and thus
\begin{equation}
  \label{eq:SumDelta}
  \sum_{i,j=0}^{N-1}\Delta(i-j)=N\,\left[\Delta(0)+2\sum_{\ell=1}^{N-1}(1-\ell/N)\;\Delta(\ell)\right]\;.
\end{equation}
Note that for all $\ell\geq2$ and $U,U',V,V'\in [0,1]^4$, we have
$$
\ell-1\leq |U-U'+\ell|\leq \ell+1\;.
$$
Hence, for all $\ell\geq2$,
\begin{equation}
  \label{eq:EquivDelta}
\gamma^2\;(\ell+1)^{-\alpha_1-\alpha_2}\leq \Delta(\ell)\leq \gamma^2\;(\ell-1)^{-\alpha_1-\alpha_2}\;.
\end{equation}
We now consider two cases.

\medskip

\noindent{\it Case (1)~:} Suppose $\alpha_1+\alpha_2\geq
1$. We get from~(\ref{eq:SumDelta}) and~(\ref{eq:EquivDelta}) that
$$
\sum_{i,j=0}^{N-1}\Delta(i-j)\leq
N \left[\Delta(0)+2\Delta(1)+2\gamma^2\sum_{\ell=2}^{N-1}(\ell-1)^{-\alpha_1-\alpha_2}\right] = O\left(N\log(N)^{\varepsilon(\alpha_1+\alpha_2)}\right)\;.
$$
The bound~(\ref{e:ineq:lem:2}) follows.

\medskip

\noindent{\it Case (2)~:} We now assume that $\alpha_1+\alpha_2<1$. In this
case, using that, as $N\to\infty$,
\begin{align*}
\sum_{\ell=2}^{N-1}(\ell+1)^{-\alpha_1-\alpha_2}  &=
\int_1^Nu^{-\alpha_1-\alpha_2} \rmd u + O(1) \;,\\
\sum_{\ell=2}^{N-1}\ell(\ell+1)^{-\alpha_1-\alpha_2}  &=
\int_1^N(u-1)u^{-\alpha_1-\alpha_2} \rmd u + O(1)\;,
\end{align*}
we get
\begin{align*}
\sum_{\ell=2}^{N-1}(1-\ell/N)\;(\ell+1)^{-\alpha_1-\alpha_2}
& = \int_1^Nu^{-\alpha_1-\alpha_2} \rmd u -\frac1N\int_1^N(u-1)u^{-\alpha_1-\alpha_2} \rmd u + O(1)\\
& =  \frac{N^{1-\alpha_1-\alpha_2}}{1-\alpha_1-\alpha_2}-\frac{N^{1-\alpha_1-\alpha_2}}{2-\alpha_1-\alpha_2}+O(1)\\
& \sim
 \frac{N^{1-\alpha_1-\alpha_2}}{(1-\alpha_1-\alpha_2)(2-\alpha_1-\alpha_2)}\;.
\end{align*}
Similarly, using instead that, as $N\to\infty$,
\begin{align*}
\sum_{\ell=2}^{N-1}(\ell-1)^{-\alpha_1-\alpha_2}  &=
\int_1^Nu^{-\alpha_1-\alpha_2} \rmd u + O(1) \;,\\
\sum_{\ell=2}^{N-1}\ell(\ell-1)^{-\alpha_1-\alpha_2}  &=
\int_1^N(u+1)u^{-\alpha_1-\alpha_2} \rmd u + O(1)\;,
\end{align*}
we get the same asymptotic equivalence, namely,
$$
\sum_{\ell=2}^{N-1}(1-\ell/N)\;(\ell-1)^{-\alpha_1-\alpha_2}\sim
 \frac{N^{1-\alpha_1-\alpha_2}}{(1-\alpha_1-\alpha_2)(2-\alpha_1-\alpha_2)}\;.
$$
Hence, with~(\ref{eq:SumDelta}) and~(\ref{eq:EquivDelta}), we get~(\ref{e:lim:lem:2}).
\end{proof}
\begin{lemma}\label{lem:cv}
Let $\alpha_1,\alpha_2,\alpha_3,\alpha_4\in (0,1)$.
\begin{enumerate}
\item Assume that
\[
\alpha_1+\alpha_2+\alpha_3+\alpha_4<3\;.
\]
Then
\begin{equation}\label{e:int1}
\int_{[0,1]^4} |u_1-u_{2}|^{-\alpha_1}|u_2-u_{3}|^{-\alpha_2}|u_3-u_{4}|^{-\alpha_3}|u_4-u_{1}|^{-\alpha_i}\rmd u_1\rmd u_2\rmd u_3\rmd u_4\;,
\end{equation}
is finite.
\item Let $\varepsilon\in\{-1,1\}$, then,
\begin{equation}\label{e:int2}
\int_{[0,1]^4}|u_1-u_{2}+\varepsilon|^{-\alpha_1}|u_2-u_{3}+\varepsilon|^{-\alpha_2}|u_3-u_{4}|^{-\alpha_3}|u_4-u_{1}|^{-\alpha_i}\rmd u_1\rmd u_2\rmd u_3\rmd u_4\;,
\end{equation}
is finite.
\end{enumerate}
\end{lemma}
\begin{proof}
We shall apply the power counting theorem in~\cite{taqqu:terrin:1991}, in particular Corollary~1 of this paper. Since the exponents are $-\alpha_i>-1$, $i=1,\cdots,4$, we need only to consider non--empty padded subsets of the set
\[
T=\{u_1-u_2,\,u_2-u_3,\,u_3-u_4,u_4-u_1\}\;.
\]
A set $W\subset T$ is said to be ``padded'' if for every element $M$ in $W$, $M$ is also a linear combination of elements in $W\setminus\{M\}$. That is, $M$ can be obtained as linear combination of other elements in $W$.
Since $T$ above is the only non--empty padded set and since
\[
d_0(T)=\mathrm{rank}(T)+\sum_{T}(-\alpha_i)=3-\sum_{i=1}^4\alpha_i>0\;,
\]
we conclude that the integral~(\ref{e:int1}) converges. This completes the proof of Part (1) of Lemma~\ref{lem:cv}.

The proof of Part (2) of Lemma~\ref{lem:cv} is even simpler since there is no padded subsets of $T$ and thus the integral~(\ref{e:int2}) always converges.
\end{proof}
\appendix
\section{Multiple Wiener-It\^o Integrals }\label{s:appendix}
Let $B=(B_{t})_{t\in \mathbb{R}}$ be a classical Wiener process on a
probability space $\left(\Omega,{\mathcal{F}},\mathbf{P}\right)$. If
$f\in L^{2}(\mathbb{R}^{n})$ with $n\geq 1$ integer, we introduce the
multiple Wiener-It\^{o} integral of $f$ with respect to $B$. The
basic reference is the monograph~\cite{nualart:2006}. Let $f\in
{\mathcal{S}_{n}}$ be an elementary symmetric function with $n$ variables that
can be written as $ f=\sum_{i_{1},\ldots ,i_{n}}c_{i_{1},\ldots
,i_{n}}1_{A_{i_{1}}\times \ldots \times A_{i_{n}}}$, where the
coefficients satisfy $c_{i_{1},\ldots ,i_{n}}=0$ if two indexes
$i_{k}$ and $i_{l}$ are equal and the sets $A_{i}\in
{\mathcal{B}}(\mathbb{R})$ are pairwise disjoint. For  such a step function
$f$ we define
\begin{equation*}
I_{n}(f)=\sum_{i_{1},\ldots ,i_{n}}c_{i_{1},\ldots
,i_{n}}B(A_{i_{1}})\ldots B(A_{i_{n}})
\end{equation*}
where we put $B(A)=\int_{\mathbb{R}} 1_{A}(s)dB_{s}$. It can be seen that
the application $ I_{n}$ constructed above from ${\mathcal{S}}_{n}$
to $L^{2}(\Omega )$ is an isometry on ${\mathcal{S}}_{n}$  in the
sense
\begin{equation}
\mathbb{E}\left[ I_{n}(f)I_{m}(g)\right] =n!\langle
f,g\rangle_{L^{2}(T^{n})}\mbox{ if }m=n  \label{e:isom}
\end{equation}
and
\begin{equation*}
\mathbb{E}\left[ I_{n}(f)I_{m}(g)\right] =0\mbox{ if }m\not=n.
\end{equation*}
\noindent Since the set ${\mathcal{S}_{n}}$ is dense in
$L^{2}(\mathbb{R}^{n})$ for every $n\geq 1$ the mapping $ I_{n}$ can be
extended to an isometry from $L^{2}(\mathbb{R}^{n})$ to $L^{2}(\Omega)$ and
the above properties hold true for this extension.

One has $I_{n}(f)=I_{n}\left(\tilde{f}\right)$,
where $\tilde{f} $ denotes the symmetrization of $f$ defined by
$$\tilde{f}(x_{1}, \ldots , x_{n}) =\frac{1}{n!}
\sum_{\sigma}f(x_{\sigma (1) }, \ldots , x_{\sigma (n) } ),$$
$\sigma$ running over all permutations of $\left\{1,\cdots,n\right\}$. Thus
\begin{equation}\label{e:ineqf}
\mathbb{E}\left[I_n(f)^2\right]=\mathbb{E}\left[I_n(\tilde{f})^2\right]=n!\|\tilde{f}\|_2^2\leq n!\|f\|_2^2\;.
\end{equation}
We will need the general formula for calculating products of Wiener
chaos integrals of any orders $m,n$ for any symmetric integrands
$f\in L^{2}(\mathbb{R}^{m})$ and $g\in L^{2}(\mathbb{R}^{n})$, which is
\begin{equation}
I_{m}(f)I_{n}(g)=\sum_{k=0}^{m\wedge n}k!\binom{m}{k}\binom{n}{k}%
I_{m+n-2k}(f\otimes_{k}g) \label{e:product},
\end{equation}
where the contraction $f\otimes_{k}g$ is  defined by
\begin{eqnarray}
 &&(f\otimes_{k} g)(s_{1},\ldots, s_{m-k},t_{1},\ldots,t_{n-k})\nonumber\\
&& =\int_{\mathbb{R}^{k}}f(s_{1},\ldots,s_{m-k}, u_{1},
\ldots,u_{k})g(t_{1},\ldots,t_{n-k},u_{1},\ldots,u_{k})
\rmd u_{1}\ldots\rmd u_{k}\;. \label{e:contra}
\end{eqnarray}
Note that the contraction $(f\otimes_{k} g) $ is an element of
$L^{2}(\mathbb{R}^{m+n-2k})$ but it is not necessarily symmetric. We will
denote its symmetrization by $(f \tilde{\otimes}_{k} g)$.

\bigskip
\noindent {\bf Acknowledgments.}
M. Clausel's research was partially supported by the PEPS project \emph{AGREE} and LabEx \emph{PERSYVAL-Lab} (ANR-11-LABX-0025-01) funded by the French program Investissement d'avenir.
Francois Roueff's research was
partially supported by the ANR project \emph{MATAIM} NT09 441552. Murad~S.~Taqqu
was supported in part by the NSF grants DMS--1007616
and DMS-1309009 at Boston University.
Ciprian Tudor's research was  by the CNCS grant PN-II-ID-PCCE-2011-2-0015 (Romania).

\bibliographystyle{plain}
\bibliography{ciprian}
\end{document}